\documentclass[journal,onecolumn]{IEEEtran}

\usepackage{amsmath, amssymb, bbm, xspace}
\usepackage{epsfig}
\usepackage{longtable}
\usepackage{color}
\usepackage{mathrsfs}

\usepackage{courier}

\newtheorem{theorem}{Theorem}[section]

\newtheorem{lemma}{Lemma}[section]

\newtheorem{proposition}[theorem]{Proposition}

\newtheorem{corollary}[theorem]{Corollary}

\newtheorem{definition}{Definition}[section]

\def\bkE{{\rm I\kern-.17em E}}
\def\bk1{{\rm 1\kern-.17em l}}
\def\bkD{{\rm I\kern-.17em D}}
\def\bkR{{\rm I\kern-.17em R}}
\def\bkP{{\rm I\kern-.17em P}}

\def\bkZ{{\bf{Z}}}

\def\bkE{{\rm I\kern-.17em E}}
\def\bk1{{\rm 1\kern-.17em l}}
\def\bkD{{\rm I\kern-.17em D}}
\def\bkR{{\rm I\kern-.17em R}}
\def\bkP{{\rm I\kern-.17em P}}

\def\bkZ{{\bf{Z}}}
\def\b12{(\beta_1,\beta_2)}

\newenvironment{proof}[1][]{{\noindent \bf Proof #1: }}{\hfill \qed \vspace{6pt}\\ }

\newcounter{example}
\renewcommand{\theexample}{\thesection.\arabic{example}}

\newcounter{remark}
\renewcommand{\theremark}{\thesection.\arabic{remark}}

\newenvironment{remark}{{\noindent \it Remark: }}{\hfill $\square$}

\newlength{\noteWidth}
\long\def\notes#1{\ifinner
{\tiny #1}
\else
\marginpar{\parbox[t]{\noteWidth}{\raggedright\tiny #1}}
\fi\typeout{#1}}

 \def\notes#1{\typeout{read notes: #1}} 

\newcommand{\minimize}[1]{\displaystyle\minim_{#1}}
\newcommand{\minim}{\mathop{\hbox{\rm min}}}
\newcommand{\maximize}[1]{\displaystyle\maxim_{#1}}
\newcommand{\maxim}{\mathop{\hbox{\rm max}}}

\def\subject{\hbox{\rm s.t.}}

\def\spose#1{\hbox to 0pt{#1\hss}}

\def\text #1{\hbox{\quad#1\quad}}

\def\nthinsp{\mskip -2   mu}

\def\superstar{^{\raise 0.5pt\hbox{$\nthinsp *$}}}
\def\SUPERSTAR{^{\raise 0.5pt\hbox{$*$}}}

\def\lamstarT {\lambda^{\raise 0.5pt\hbox{$\nthinsp *$}T}}

\def\SOL{{\rm SOL}}

\let\forallnew\forall
\renewcommand{\forall}{\forallnew\ }
\let\forall\forallnew

		\def\bkE{{\rm I\kern-.17em E}}
		\def\bk1{{\rm 1\kern-.17em l}}
		\def\bkD{{\rm I\kern-.17em D}}
		\def\bkR{{\rm I\kern-.17em R}}
		\def\bkP{{\rm I\kern-.17em P}}
		\def\bkY{{\bf \kern-.17em Y}}
		\def\bkZ{{\bf \kern-.17em Z}}
		\def\bkC{{\bf  \kern-.17em C}}

{\begin{list}{}%
         {\setlength{\leftmargin}{#1}}%
         \item[]%
}
{\end{list}}

		\def\bsp{\begin{split}}
		\def\beq{\begin{eqnarray}}
		\def\bal{\begin{align*}}
		\def\bc{\begin{center}}
		\def\be{\begin{enumerate}}
		\def\bi{\begin{itemize}}
		\def\bs{\begin{small}}
		\def\bS{\begin{slide}}
		\def\ec{\end{center}}
		\def\ee{\end{enumerate}}
		\def\ei{\end{itemize}}
		\def\es{\end{small}}
		\def\eS{\end{slide}}
		\def\eeq{\end{eqnarray}}
		\def\eal{\end{align*}}
		\def\esp{\end{split}}
		\def\qed{ \vrule height7.5pt width7.5pt depth0pt}  

		\def\problem#1#2#3#4{\fbox
		 {\begin{tabular*}{0.80\textwidth}
			{@{}l@{\extracolsep{\fill}}l@{\extracolsep{6pt}}l@{\extracolsep{\fill}}c@{}}
				#1 & $\minimize{#2}$ & $#3$ & $ $ \\[5pt]
					 & $\subject\ $    & $#4$ & $ $
			\end{tabular*}}
			}

\def\maxproblem#1#2#3#4{\fbox
		 {\begin{tabular*}{0.80\textwidth}
			{@{}l@{\extracolsep{\fill}}l@{\extracolsep{6pt}}l@{\extracolsep{\fill}}c@{}}
				#1 & $\maximize{#2}$ & $#3$ & $ $ \\[5pt]
					 & $\subject\ $    & $#4$ & $ $
			\end{tabular*}}
			}

	\def\cp2problem#1#2#3#4{\fbox
		 {\begin{tabular*}{0.9\textwidth}
			{@{}l@{\extracolsep{\fill}}l@{\extracolsep{6pt}}l@{\extracolsep{\fill}}c@{}}
				#1 & & $#4 $ 
			\end{tabular*}}}

		\def\bkE{{\rm I\kern-.17em E}}
		\def\bk1{{\rm 1\kern-.17em l}}
		\def\bkD{{\rm I\kern-.17em D}}
		\def\bkR{{\rm I\kern-.17em R}}
		\def\bkP{{\rm I\kern-.17em P}}
		
		\def\bkZ{{\bf{Z}}}

\newcommand {\beeq}[1]{\begin{equation}\label{#1}}
\newcommand {\eeeq}{\end{equation}}
\newcommand {\bea}{\begin{eqnarray}}
\newcommand {\eea}{\end{eqnarray}}

\def\texitem#1{\par\smallskip\noindent\hangindent 25pt
               \hbox to 25pt {\hss #1 ~}\ignorespaces}

\usepackage{graphicx}
 \usepackage{amssymb}
 \usepackage{amsmath}
 \usepackage{color}
\def\subject{\hbox{\rm subject to}}
\def\SOL{{\rm SOL}}

\def\maxproblem#1#2#3#4{\fbox
	{\begin{tabular*}{0.95\textwidth}
			{@{}l@{\extracolsep{\fill}}l@{\extracolsep{6pt}}l@{\extracolsep{\fill}}c@{}}
			#1 & $\maximize{#2}$ & $#3$ & $ $ \\[5pt]
			& $\subject\ $    & $#4$ & $ $
		\end{tabular*}}
	}
	
\def\problem#1#2#3#4{\fbox
		{\begin{tabular*}{0.95\textwidth}
				{@{}l@{\extracolsep{\fill}}l@{\extracolsep{6pt}}l@{\extracolsep{\fill}}c@{}}
				#1 & $\minimize{#2}$ & $#3$ & $ $ \\[5pt]
				& $\subject\ $    & $#4$ & $ $
			\end{tabular*}}
		}

\usepackage{cite}

\ifCLASSINFOpdf
\else
\fi
\usepackage{algorithmic}

\usepackage{array}

\begin{document}

%
\title{Price-Coupling Games and the Generation Expansion Planning Problem }
%
%
%

\author{Mathew~P.~Abraham and 
        Ankur~A.~Kulkarni
\thanks{Mathew  and Ankur are with the Systems and Control Engineering group, Indian Institute of Technology Bombay, Mumbai, India, 400076. email: \texttt{mathewp@iitb.ac.in}, \texttt{kulkarni.ankur@iitb.ac.in}.
A preliminary version of the paper is presented at Indian Control Conference 2018, Kanpur~\cite{abraham2018existence}}
}

\maketitle

\begin{abstract}In this paper, we introduce and study a class of games called price-coupling games that arise in many scenarios, especially in the electricity industry. In a price-coupling game, there is a part of the objective function of a player which has an identical form for all players and there is coupling in the cost functions of players through a price which is determined uniformly for all players by an independent entity called the \textit{price-determining player} (e.g. independent system operator (ISO) in an electricity market). This price appears in the objective function only in the part which is identical for all players.  We study the existence of equilibria in such games under two broad categories, namely \textit{price-anticipative} and \textit{price-taking} formulations. In the price-anticipative formulation, the players anticipate the price and make their decisions while in the price-taking formulation, the players make their decisions considering the price as a given parameter. We model the price-anticipative case as a leader-follower formulation where the players (leaders) conjecture the price (follower's decision) and make their decision. The price-taking formulation is modelled as an $N+1$ player game with the additional player as the price-determining player. The existence of an equilibrium in such games are not easy mainly because of the coupled-constraint structure of the game and the non-convexity of the action set. We give conditions for the existence of equilibria in both formulations.  We apply our results to  analyze the existence of an equilibrium in the generation expansion planning problem using the above results.      
    
\end{abstract} 


\section{Introduction}\label{sec:introduction}
The recent past has seen the stupendous growth of market mechanisms involving buyers and sellers in markets as complex as commodity markets (like electricity) or service markets like transportation. Unlike in old-school markets, the price that suppliers see in these modern markets is itself determined by a fairly complicated decision problem. For example, in the electricity industry, the price is determined by the \textit{independent system operator} such a way so as to procure electricity at the lowest cost for the consumer. In on-demand transportation, platforms such as Uber, Lyft or Ola determine prices conceivably based on a variety considerations such, demand arrival, supply availability, tolls, etc. This is leading to novel decision problems on the supply side for which classical analysis fails to apply. Moreover, the complexity of these mechanisms also means that this analysis is mathematically non-trivial. The goal of this paper is to develop foundational tools for claiming the existence of equilibria in such games that arise in such formulations by introducing a new class of games, called \textit{price-coupling games}.

In microeconomics,  players on the supply side are generally modeled in two ways, as competitive players and as strategic players. In the formulation as competitive players, the players are called \textit{price takers}, which means that the players take their decision assuming the price is exogenously given \cite{schiro2015perfectly}, \cite{arrow54existence}. This is indeed true in markets where there is no player that has enough market power to affect prices. While in the formulation as strategic players, the players anticipate the price and make their decision.  Thus, while making decisions, the players consider the price of the quantity to be determined as a function of their decisions \cite{hobbs-strategic}, \cite{hobbs01linear}. We call the first formulation as the \textit{price-taking formulation} and the latter as the \textit{price-anticipative formulation}.

In this paper, we consider the existence of equilibria in these formulations for the class of games called price-coupling games we introduce in this paper. In such games, the objective function of players has the following form. It is a summation of two terms, the first term which depends on the price and the other term which is independent of the price. The form of the term which is dependent on the price is same for all players and the form of the other term which is independent of price can be different for different players. The price itself is determined by another player called the \textit{price-determining player} who has his own optimization problem. As stated above, the study of such games is motivated by the observation that many games of supply-side competition arising in practice fall in this category, where the price is identical for all players and it has a uniform influence on the players. 

Our main thrust in this paper is that this class of games has a peculiar structure which one can exploit to provide results for a variety of settings concerning such games. To make our contributions precise, we introduce a game $\mathcal{G}$ with following notation.
Let $\mathcal{N}= \{1,2,\dots,N\}$ be the set of players,
$X_i \subseteq \mathbb{R}^{m_i}$, $i \in \mathcal{N}$ be the set of decisions of player $i$, $X \triangleq \prod\limits_{i \in \mathcal{N}}^{} X_i$ be the set of decision-tuples of all players, let $x_i \in X_i$ be the  decision  of player $i$, $x = (x_1,x_2,\dots,x_N) \in X$ be the decisions of all players, and $x^{-i} \triangleq (x_1,\dots,x_{i-1},x_{i+1},\dots,x_N) \in X^{-i}$ be the decisions of every player except player $i$. Finally, let $\SOL[\rm{P}]$  be solution of the optimization problem $\rm{P}$.

\subsubsection*{Player $i$'s problem in a price-coupling game}
For a game $\mathcal{G}$,  we assume player $i$'s problem denoted by $P_i(x^{-i})$ which is given as follows:
$$
\maxproblem{$P_i(x^{-i})$}
{x_i, p}
{h(p,x)+g_i(x)}
{x_i \in {X}_i, p \in \SOL[\mathcal{S}(x)]. }
$$
Here $p$ represents the price. The objective function of player $i$ is a summation of two terms, $h$ which is dependent on $p$ and $g_i$ which is independent of $p$. Also note that the part of the objective function which is dependent on $p$ has the same form in the objective function of all players. The other part which is independent of $p$ may have different forms for different players. The price in a price-coupling game is given as a solution of an optimization problem denoted by $\mathcal{S}(x),$ which is given below. We call this problem as the {price-determining player problem}. 
\subsubsection*{Price-determining player}

For a game $\mathcal{G}$, the price-determining player takes $x \in X$ from the other players and determines a price $p$ for all players. The optimization problem of the price-determining player $\mathcal{S}(x)$ in its general form is given as follows: 
$$
\problem{$\mathcal{S}(x)$}
{p}
{f(x,p)}
{p \in \mathcal{M}(x). }
$$
Here the function $f(x,\cdot)$ is the cost function that the price-determining player minimizes over $\mathcal{M}(x)$, the constraint on the price. This is an illustration of one kind of a price-coupling game. There are other price-coupling games with variations from this structure which will be discussed in later sections. 

Guaranteeing the existence of an equilibrium is not straightforward in a price-coupling game. One reason is that unlike a standard game, the set of constraints of players are coupled, that is the decision of one player changes the feasible set of other players. Another reason is the non-convexity of the feasible region of a player, which occurs due to the appearance of $\mathcal{S}(x)$ in the constraints of $P_i(x^{-i})$.

The main contributions of the paper are as follows.
A player's problem in a price-coupling game is formulated as a \textit{price-anticipative formulation} and \textit{price-taking formulation} depending on how the price determined by the price-determining player influences the participants' decision making.  In the price-anticipative formulation, the players conjecture the solution of the price-determining player problem while making their decision. Thus, we model it as a leader-follower Stackelberg game with the players considered as the leaders and the price-determining player as the follower. It is known that there need not exist an equilibrium even for a simple leader-follower game as shown by Pang and Fukushima in \cite{pang98complementarity}. In this formulation, we categorize and analyze the existence of equilibria of the price-coupling game in {two} different classes depending on the mathematical properties of the game. We provide conditions for the existence of equilibria in these classes of price-coupling games. 

The second formulation is the price-taking formulation in which the price is taken as a given parameter for each player except the price-determining player. As pointed out in a recent paper \cite{schiro2015perfectly} there has not been much progress regarding the existence of an equilibrium in price-taking formulation after the results by Arrow and Debreu in \cite{arrow54existence}. In our formulation an equilibrium in the price-taking formulation is given as an equilibrium of an $N+1$ player game with the additional player is the price-determining player. In this formulation, we use a potential function approach to provide conditions for the existence of an  equilibrium.	 Finally, we consider a concrete application -- the generation expansion planning problem and we apply our results to this application. The existence of equilibria in generation expansion planning problem presented in \cite{chuang2001game} is analyzed in both these formulations. 

Our work follows a long line of work electricity markets and various related models. The main challenges in such problems have been discussed in~\cite{pang05quasi}, \cite{kulkarni2014shared}, \cite{kulkarni2015existence}, \cite{kulkarni2013consistency}, and the recent works \cite{pozo2017basic}, \cite{zerrahn2017network}. The volume of work here is substantial and we refer the interested reader to the above works for details.
The paper is organized as follows. The next section gives an example of price-coupling games from the electricity market.  Section \ref{sec:price_anticipative} elaborates the price-anticipative formulation of the price-coupling game. The price-taking formulation and the different cases are explained in Section \ref{sec:price_taking}. Section \ref{sec:GEP} elaborates the application of results to generation expansion planning problem. The paper ends with a conclusion in Section \ref{sec:conclusion}.

\section{An Example: Cournot Model of Electricity Trading}
{We consider a well-known example of the Cournot model of electricity trading and show how it fits into our framework of a price-coupling game.} In a Cournot model of competition \cite{hobbs01linear}, the players submit their bids as quantity $q_i, i \in \mathcal{N}$ which they are willing to supply. Let $q= q_1+\dots+q_N$ be the total quantity bids of the players. Suppose $\mathcal{D}(q)$ is the demand curve of that market which gives the relation of price that the consumers are willing to pay for a $q$ quantity of electricity. The ISO (Independent system operator) determines the market clearing price from bids by the players $q_i,i \in \mathcal{N}$ and the demand curve $\mathcal{D}(q)$. Suppose $\mathbf{q}\triangleq (q_1,q_2,\dots,q_N)$. We denote the ISO's problem in the Cournot model of competition as $\mathcal{ISO}^c(\mathbf{q};\mathcal{D})$. Here the price is determined trivially from the demand curve. More generally the price may be a solution of some optimization problem denoted by $\mathcal{S}(\mathbf{q})$ as given in the introduction. Let $C_i$ be the cost for producing $q_i$ quantity of electricity for supplier $i$ and $Q_i$ be the maximum production capacity of supplier $i$. Given $q^{-i}$, the decision of rivals, the supplier $i$'s problem in the Cournot model of competition is given by the problem denoted by  $P^c_i(q^{-i})$.  
$$
\maxproblem{$P^c_i(q^{-i})$}
{{q}_i}
{p{q}_i - C_i({q}_i)}
{\begin{array}{l@{\ }c@{\ }l}
	0 \leq {q}_i \leq Q_i, \\p \in \SOL [\mathcal{ISO}^c(\mathbf{q};\mathcal{D})] .
	\end{array}}
$$
It can be seen that this game has a structure of a price-coupling game. That is, the part of the objective function which is dependent on the price is the revenue $p.{q}_i$ for player $i$ and this part has the same form for all players and the price $p$ which is determined by the ISO is identical for all the players. Also, the term which is independent of price is the cost of generation which may vary for players. However this game has a structure which is distinct from the structure discussed in the introduction. These distinctions will matter in the results and are discussed more precisely in Section \ref{sec:price_anticipative}.

\section{{Price-Anticipative Formulation}} \label{sec:price_anticipative}
In this section we consider the price-anticipative formulation of a price-coupling game. As mentioned in Section \ref{sec:introduction}, in a price-anticipative formulation, the price-coupling game is modelled as a multi-leader single-follower game with players as the leaders and the price-determining player as the follower. 
For player $i \in \mathcal{N}$, consider a price-coupling game with the player $i$'s problem denoted by $P_i(x^{-i})$. In the price-anticipative formulation, the player $i$ conjectures the price determined by the price-determining player. Hence, we denote the price as $p_i$, which is dependent on player $i$. We define the  feasible set of player $i$ in price-anticipative formulation as $\Omega_i(x^{-i}) \triangleq \{(x_i,p_i)| x_i \in {X}_i, p_i \in \SOL[\mathcal{S}(x)] \}.$ With these considerations, the player $i$'s problem in price-anticipative formulation can be rewritten as $P^{a\mathbf{1}}_i(x^{-i})$. 
$$
\maxproblem{$P^{a\mathbf{1}}_i(x^{-i})$}
{(x_i, p_i) \in \Omega_i(x^{-i})}
{h(p_i,x)+g_i(x).}
{x_i \in {X}_i, p \in SOL[\mathcal{S}(x)]. }
$$
We use $a\mathbf{1}$ in $P^{a \mathbf{1}}_i(x^{-i})$ to identify that this is the first class of price-coupling games that we consider in the price-anticipative formulation. We use $\mathcal{E}_1$ to denote the class of price-coupling games in the price-anticipative formulation with player $i$'s problem given by  $P^{a\mathbf{1}}_i(x^{-i})$.

\subsection{Existence of equilibria for game $\mathcal{E}_1$}
Define
$\mathcal{F}^1 \triangleq \{(x,p)|x_i \in {X}_i, i \in \mathcal{N}, p \in \SOL[\mathcal{S}(x)] \}.
$
We define an equilibrium for game $\mathcal{E}_1$ as follows: 
\begin{definition}\label{def:price_anticipative}
	A point $(x^*,p^*) \in \mathcal{F}^1$ is said to be an equilibrium for game $\mathcal{E}_1$, if $\forall i \in \mathcal{N}$,
	\begin{align*} 
	{h(p^*,x^{*})+g_i(x^{*})} \geq {h(p_i,x_i,x^{-i*})+g_i(x_i,x^{-i*})}, \quad \forall (x_i,p_i) \in \Omega_i(x^{-i*}).
	\end{align*}
\end{definition}

To show the existence of an equilibrium we require that the functions $g_i, i \in \mathcal{N}$ admit a potential function. We recall the definition of a potential function from \cite{monderer96potential} which is given as follows:
\begin{definition}{\label{def:potential_function}}
	The functions $g_i : X \rightarrow \mathbb{R}, i \in \mathcal{N}$ is said to admit a potential function $\pi: X \rightarrow \mathbb{R}$ if $\forall i \in \mathcal{N}$,
	\begin{align*}
	g_i(\tilde{x}_{i},x^{-i}) - g_i(\hat{x}_{i},x^{-i}) = \pi(\tilde{x}_{i},x^{-i}) - \pi(\hat{x}_{i},x^{-i}), \quad \forall x^{-i}\in  X^{-i}, \forall (\tilde{x}_{i},\hat{x}_{i}) \in X_i.
	\end{align*}
\end{definition}
In order to show the existence of equilibria for the class of games $\mathcal{E}_1$,  we need a result which is given by the following lemma.
\begin{lemma}\label{lemma:1}
	A point $(x_i,p)$ is feasible for player $i$'s problem $P^{a\mathbf{1}}_i (x^{-i})$ if and only if $(x,p) \in \mathcal{F}^1$. That is,
	$$(x_i,p) \in \Omega_i(x^{-i}) \iff (x,p) \in \mathcal{F}^1.$$
\end{lemma}

\begin{proof}
	``$\Rightarrow$"
	For a given $x^{-i} \in X^{-i}$, consider a point $(x_i,p) \in \Omega_i(x^{-i})$. That means, $x_i \in X_i, p \in \SOL[\mathcal{S}(x)]$. Now by combining $x_i \in X_i$ and $x^{-i} \in X^{-i}$ we can rewrite the above equations as $x \in X, p \in \SOL[\mathcal{S}(x)]$. Hence $(x,p) \in \mathcal{F}^1$.
	
	``$\Leftarrow$" Suppose $(x,p) \in \mathcal{F}^1$. That is, $x \in X, p \in \SOL[\mathcal{S}(x)]$. For some $i\in \mathcal{N}$, we separate the decisions of player $i$  and adversaries as $x_i \in X_i$ and $x^{-i} \in X^{-i}$. For some fixed  $x^{-i}$, the equations can be rewritten as $x_i \in X_i, p \in \SOL[\mathcal{S}(x)]$. Hence, $(x_i,p)\in \Omega_i (x^{-i})$.
\end{proof}
Now for the game $\mathcal{E}_1$, consider the following optimization problem denoted by $\mathbb{P}^1$.
$$
\maxproblem{$\mathbb{P}^1$}
{(x, p) \in \mathcal{F}^1}
{h(p,x)+ \pi(x).}
{(x, p) \in \mathcal{F}^1. }
$$	
Theorem $\ref{th:price_anticipative1}$ gives a relation of an equilibrium of the game $\mathcal{E}_1$ to the solution of the optimization problem $\mathbb{P}^1$. From this relation, we provide conditions for the existence of an equilibria for game $\mathcal{E}_1$ in Corollary $\ref{Coro:1}$.
\begin{theorem} \label{th:price_anticipative1}
	Consider the game denoted by $\mathcal{E}_1$. Suppose the functions $g_i, i \in \mathcal{N}$ admit a potential function $\pi$. For this game, consider an optimization problem denoted by $\mathbb{P}^1$. Suppose $(x^*,p^*) \in \mathcal{F}^1$ is a maximizer of the problem $\mathbb{P}^1$, then $(x^*,p^*) \in \mathcal{F}^1$ is also an equilibrium of the game $\mathcal{E}_1$.
\end{theorem}
\begin{proof}
	Suppose $(x^*,p^*) \in \mathcal{F}^1$ is a maximizer of the problem $\mathbb{P}^1$. Then,
	$$
	{h(p^*,x^{*})+\pi(x^{*})} \geq  { h(p,x)+\pi(x)}, \quad \forall (x,p) \in \mathcal{F}^1.
	$$
	That is, for some $i \in \mathcal{N}$,	
	$$
	{h(p^*,x^{*})+\pi(x^{*})} \geq  { h(p,x_i,x^{-i})+\pi(x_i,x^{-i})}, \forall (x,p) \in \mathcal{F}^1.
	$$
	These inequalities hold true even if $x^{-i}= x^{-i*}$. Thus for $ i \in \mathcal{N}$ we can rewrite the inequalities as,
	\begin{align*}
	{h(p^*,x^{*})+\pi(x^{*})} \geq  { h(p,x_i,x^{-i*})+\pi(x_i,x^{-i*})}, \quad  \forall (x_i,x^{-i*},p) \in \mathcal{F}^1.
	\end{align*}
	By Lemma $\ref{lemma:1}$, for $ i \in \mathcal{N}$,
	\begin{align*}
	{h(p^*,x^{*})+\pi(x^{*})} \geq  { h(p,x_i,x^{-i*})+\pi(x_i,x^{-i*})}, \quad  \forall (x_i,p) \in \Omega_i(x^{-i*}).
	\end{align*}
	Since $\pi$ is a potential function for $g_i, i \in \mathcal{N}$,
	\begin{align*}
	{h(p^*,x^{*})+g_i(x^{*})} \geq  { h(p,x_i,x^{-i*})+g_i(x_i,x^{-i*})}, \quad  \forall (x_i,p) \in \Omega_i(x^{-i*}).
	\end{align*}
	Since we consider an arbitrary $i \in \mathcal{N}$, this condition is valid for all $i \in \mathcal{N}$. Hence $(x^*,p^*) \in \mathcal{F}^1$ is an equilibrium of the price-coupling game $\mathcal{E}_1$.
\end{proof}
Since the maximizer of the problem $\mathbb{P}^1$ is an equilibrium of game $\mathcal{E}_1$, condition for the existence of equilibria is given as the condition for the existence of a solution for the optimization problem  $\mathbb{P}^1$ which is given in the following corollary.
\begin{corollary} \label{Coro:1}
	For a price-coupling game $\mathcal{E}_1$ with $g_i, i \in \mathcal{N}$ admitting a potential function $\pi$, admits an equilibrium if the function, ${h+ \pi}$ is continuous and the set $\mathcal{F}^1$ is non empty and compact. The set  $\mathcal{F}^1$ is compact if the {set $X$ is compact, the function $f$ is continuous and the set valued function $\mathcal{M}$ is continuous and uniformly bounded.}
\end{corollary}
{We omit the detailed proof, but we give an outline of the proof. The first part follows from Theorem $\ref{th:price_anticipative1}$ and Weirstrass theorem. By Theorem $8$ in \cite{hogan73point}, since $\mathcal{M}$ and ${f}$ are continuous, $\mathcal{F}^1$ is closed. Also, $\mathcal{F}^1$ is bounded since $\mathcal{M}$ is uniformly bounded and $X$ is compact. Hence $\mathcal{F}^1$ is compact.}
Thus, to guarantee the existence of an equilibrium for a price-coupling game $\mathcal{E}_1$, one can check the conditions given in the Corollary $\ref{Coro:1}$. 

\begin{remark}
	Suppose the player dependent term $g_i$ in the objective function of a player is dependent only on $x_i$, and is independent of $x^{-i}$, then a potential function of the terms $g_i, i \in \mathcal{N}$ is given by $\pi = \sum_{i \in \mathcal{N}}^{} g_i$.
\end{remark}

\subsection{Shared constraints and existence of equilibria for game $\mathcal{E}_2$}

In this section, we consider the price-coupling game in price-anticipative formulation with the player $i$'s problem given by $P^{a\mathbf{2}}_i(x^{-i})$. The difference of $P^{a\mathbf{2}}_i(x^{-i})$ from $P^{a\mathbf{1}}_i(x^{-i})$ is that the identical function part $h$ is dependent on $p$ and $x_i$ and is independent of $x^{-i}$. We use $\mathcal{E}_2$ to denote this class of price-coupling games.
$$
\maxproblem{$P^{a\mathbf{2}}_i(x^{-i})$}
{(x_i, p_i) \in \Omega_i(x^{-i})}
{h(p_i,x_i)+g_i(x).}
{x_i \in {X}_i, p \in SOL[\mathcal{S}(x)]. }
$$ An equilibrium for the price-coupling game $\mathcal{E}_2$ is defined as follows:
\begin{definition}\label{def:price_anticipative3}
	A point $(x^*,p^*) \in \mathcal{F}^1$ is said to be an equilibrium for the game $\mathcal{E}_2$, if $\forall i \in \mathcal{N}$,
	\begin{align*} 
	{h(p^*,x_i^{*})+g_i(x^{*})} \geq {h(p_i,x_i)+g_i(x_i,x^{-i*})}, \quad \forall (x_i,p_i) \in \Omega_i(x^{-i*}).
	\end{align*}
\end{definition}
Recall that we have considered $p_i$ as the price conjectured by player $i$. Thus a player can make a conjecture independent of other players. To give a result for the existence of equilibria for the game $\mathcal{E}_2$, we modify the game with an additional constraint. The additional constraint in a player's problem is that the price conjectured by the players need to be consistent, i.e. $p_i=p_j, \forall i,j \in \mathcal{N}$. We call this constraint as the consistent conjecture of the price. With this additional constraint, we define the feasible set of a player $i$ as the set denoted by $\bar{\Omega}_i(x^{-i},p^{-i})$.
\begin{align}\label{eq:feasible_set_modified}
\bar{\Omega}_i(x^{-i},p^{-i}) = \{(x_i,p_i)| x_i \in X_i, p_i \in \SOL[\mathcal{S}(x)], p_i=p_j, \forall j \in \mathcal{N}\backslash\{i\}\}.
\end{align}
Let $\bar{\Omega}(x,p) \triangleq \prod\limits_{i=1}^{N} \bar{\Omega}_i(x^{-i},p^{-i}).$ With the additional constraint of the consistent conjecture of price, we redefine the player $i$'s problem in game $\mathcal{E}_2$ as $\bar{P}^{a\mathbf{2}}_i(x^{-i},p^{-i})$ which is given as follows:
$$
\maxproblem{$\bar{P}^{a\mathbf{2}}_i(x^{-i},p^{-i})$}
{(x_i, p_i) \in \bar{\Omega}_i(x^{-i},p^{-i})}
{h(p_i,x_i)+g_i(x).}
{x_i \in {X}_i, p \in SOL[\mathcal{S}(x)]. }
$$
We call this modified game as a \textit{price-coupling game with consistent price conjectures} and denote it by $\bar{\mathcal{E}}_2$. For defining an equilibrium in the game $\bar{\mathcal{E}}_2$, consider the set $\mathcal{F}^2$ which is defined as follows. 
\begin{align*}
\mathcal{F}^2 \triangleq \{(x,p_1,\dots,p_N)|x \in {X}, p_i \in \SOL[\mathcal{S}(x)], i \in \mathcal{N}, p_i=p_j, \forall i,j \in \mathcal{N} \}.
\end{align*}
Now we define an equilibrium in the game $\bar{\mathcal{E}}_2$ as follows.
\begin{definition}\label{def:price_anticipative}
	A point $(x^*,p^*) \in \mathcal{F}^2$ is said to be an equilibrium of the game $\bar{\mathcal{E}}_2$ if $\forall i \in \mathcal{N}$,
	\begin{align*} 
	{h(p^*,x_i^{*})+g_i(x^{*})} \geq {h(p_i,x_i)+g_i(x_i,x^{-i*})}, \forall (x_i,p_i) \in \bar{\Omega}_i(x^{-i*},p^{-i*}).	\end{align*}
\end{definition}
The following lemma gives a relation between the original price-coupling game $\mathcal{E}_2$ and the price-coupling game with consistent price conjecture $\bar{\mathcal{E}}_2$.
\begin{lemma}{\label{lemma:relation}}
	Consider a game $\mathcal{E}_2$. Suppose $(x^*,p^*) \in \mathcal{F}^1$ is an equilibrium of this game. Then $(x^*,p^*) \in \mathcal{F}^2$ is also an equilibrium of the game $\bar{\mathcal{E}}_2$ \cite{kulkarni2013consistency}.
\end{lemma}
\begin{proof}
	Since at an equilibrium of the game $\mathcal{E}_2$, the price conjectured by the players are the same, an equilibrium of $\mathcal{E}_2$ is also an equilibrium of $\bar{\mathcal{E}}_2$ . 
\end{proof}

To give conditions for the existence of an equilibria we use the \textit{shared constraint} structure of the price-coupling game with consistent conjecture \cite{abraham2017newresults},\cite{kulkarni2013consistency}. In a coupled constraint game, the action of a player restricts the action set of other players. Moreover, if the coupled action set in a game is same for all players, then we call it as a shared constraint game. In our case, the set mapping $\bar{\Omega}(x,p)$ is said to be a shared constraint mapping if there exists a set $\mathcal{W}$ such that,
\begin{equation*}
({x}_{i},p_i) \in \bar{\Omega}_{i}({x^{-i},p^{-i}}) \Longleftrightarrow (x,p)\in \mathcal{W}.
\end{equation*}
Shared constraint games were introduced by Rosen~\cite{rosen65existence} and have attracted significant study over the last decade~\cite{facchinei07ongeneralized}, 
\cite{kulkarni09refinement}, \cite{kulkarni2014shared}, \cite{kulkarni2015existence}. The following lemma shows that for the game $\bar{\mathcal{E}}_2$, the set mapping $\bar{\Omega}(x,p)$ has a shared constraint structure.

\begin{lemma}\label{lemma:share_constraint}
	Consider a price-coupling game with consistent price conjectures $\bar{\mathcal{E}}_2$. 	There exists a shared constraint set $\mathcal{F}^2$ such that a point $(x_i,p_i)$ is feasible for player $i$'s problem $\bar{P}^{a\mathbf{2}}_i (x^{-i},p^{-i})$ if and only if $(x,p) \in \mathcal{F}^2$. That is,
	$(x_i,p_i) \in \bar{\Omega}_i(x^{-i},p^{-i}) \iff (x,p) \in \mathcal{F}^2,$ where, $\mathcal{F}^2= (X \times \mathcal{B}) \cap \mathcal{H}$, where $$
	\mathcal{B} \triangleq \{(p_1,\dots,p_N)| p_i=p_j, \forall i,j \in \mathcal{N}\},$$$$\mathcal{H} \triangleq \{(x,p_1,\dots,p_N)| x \in X, p_i \in \SOL[\mathcal{S}(x)]\}.$$
\end{lemma}
\begin{proof}
	Let $i \in \mathcal{N}$ be an arbitrary player for game $\bar{\mathcal{E}}_2$.
	Consider the constraint set of player $i$ with the condition of consistent conjecture of price denoted by ${\bar{\Omega}_{i}}({x^{-i},p^{-i}})$ which is given in \eqref{eq:feasible_set_modified}.
	By using the defined sets $\mathcal{B}$ and $\mathcal{H}$, the set ${\bar{\Omega}_{i}}({x^{-i},p^{-i}})$ can be redefined as,
	\begin{align*}
	\bar\Omega_{i}({x^{-i},p^{-i}}) = \{(\hat{x}_{i},\hat{p}_{i}) | (\hat{x}_{i},x^{-i})\in X, (\hat{p}_{i},p^{-i})\in \mathcal{B}, (\hat{x}_{i},x^{-i},\hat{p}_{i},p^{-i}) \in \mathcal{H} \}.
	\end{align*}
	i.e., $({x}_{i},{p}_{i}) \in \bar\Omega_{i}({x^{-i},p^{-i}}) \Longleftrightarrow (x,p)\in (X\times \mathcal{B})\cap \mathcal{H}$, which is independent of $i$. Since the implication holds for each $i \in \mathcal{N}$,  $ ({x}_{i},{p}_{i}) \in \bar\Omega_{i}({x^{-i},p^{-i}}) \Longleftrightarrow (x,p)\in \mathcal{F}^2,\forall i \in \mathcal{N},$ as required for a shared constraint. 
\end{proof}
Thus $\mathcal{F}^2$ is a shared constraint for the game $\bar{\mathcal{E}}_2$. For the game $\bar{\mathcal{E}}_2$, consider the following problem denoted by  $\mathbb{P}^2.$
$$
\maxproblem{$\mathbb{P}^2$}
{(x,p) \in \mathcal{F}^2}
{ \sum\limits_{i\in \mathcal{N}} \bigl[h(p_i,x_i)\bigr]+ \pi(x).}
{(x,p) \in \mathcal{F}^2 }
$$
The following theorem relates the solution of the  problem $\mathbb{P}^2$ to an equilibrium of the game $\bar{\mathcal{E}}_2$.
\begin{theorem} \label{theorem:potential}
	Consider the game $\bar{\mathcal{E}}_2$. Suppose $g_i,i\in \mathcal{N}$ admits a potential function $\pi$. For this game, consider the optimization problem denoted by $\mathbb{P}^2$. Suppose $(x^*,p^*) \in \mathcal{F}^2$ is a maximizer of the problem $\mathbb{P}^2$, then $(x^*,p^*) \in \mathcal{F}^2$ is also an equilibrium of $\bar{\mathcal{E}}_2$. {There exists an equilibrium for this game if the functions $\sum_{}^{}h + \pi$ and $f$ are continuous, the set $X$ is compact and the set valued function $\mathcal{M}$ is continuous and uniformly bounded.}
\end{theorem}
\begin{proof}
	Let $(x^*,p^*) \in \mathcal{F}^2$ be a maximizer of the problem $\mathbb{P}^2$. Then, $\forall (x,p) \in \mathcal{F}^2$,
	\begin{align*}
	{\sum\limits_{i\in \mathcal{N}}  \bigl[h(p_i^*,x_i^{*})\bigr]+  \pi(x^{*})} \geq  {\sum\limits_{i\in \mathcal{N}}  \bigl[h(p_i,x_i)\bigr]+  \pi(x)}.
	\end{align*}
	That is, for some $i \in \mathcal{N}$,	
	\begin{align*}
	{\sum\limits_{i\in \mathcal{N}}  \bigl[h(p_i^*,x_i^{*})\bigr]+  \pi(x^{*})} \geq h(p_i,x_i) + {\sum\limits_{j\in \mathcal{N} \backslash \{i\}}  \bigl[h(p_j,x_j)\bigr]+} { \pi(x_i,x^{-i})}, \quad  \forall (x_i,x^{-i},p_i,p^{-i}) \in \mathcal{F}^2.
	\end{align*}
	These inequalities hold true even if $x^{-i}= x^{-i*}$ and $p^{-i}= p^{-i*}$. Thus for $ i \in \mathcal{N}$ we can rewrite the inequalities as,
	\begin{align*}
	{\sum\limits_{i\in \mathcal{N}}  \bigl[h(p_i^*,x_i^{*})\bigr]+  \pi(x^{*})} \geq h(p_i,x_i) + {\sum\limits_{j\in \mathcal{N} \backslash \{i\}}  \bigl[h(p^*_j,x^*_j)\bigr]+} { \pi(x_i,x^{-i*})}, \quad \forall (x_i,x^{-i*},p_i,p^{-i*}) \in \mathcal{F}^2.
	\end{align*}
	By Lemma $\ref{lemma:share_constraint}, \forall i \in \mathcal{N}$
	and by cancelling equal terms on both sides, for $ i \in \mathcal{N}$,$\quad \forall (x_i,p_i) \in \Omega_i(x^{-i*},p^{-i*}),$
	$
	{h(p_i^*,x_i^{*})+  \pi(x^{*})} \geq h(p_i,x_i) +   \pi(x_i,x^{-i*}).
	$
	Since $\pi$ is a potential function for $g_i, i \in \mathcal{N}$,$\quad \forall (x_i,p_i)$,
	$
	{h(p_i^*,x_i^{*})+  g_i(x^{*})} \geq h(p_i,x_i) +   g_i(x_i,x^{-i*}).
	$
	Since we considered an arbitrary $i \in \mathcal{N}$, this condition is valid for all $i \in \mathcal{N}$. Hence $(x^*,p^*) \in \mathcal{F}^2$ is an equilibrium of the game $\bar{\mathcal{E}}_2$.
\end{proof}
\section{Price-taking Formulation} \label{sec:price_taking}
In this section, we consider the price-taking formulation of a price-coupling game. As discussed in Section \ref{sec:introduction}, in this formulation, a player makes the decision by considering the price $p$ as a given parameter. That is, given $p$ and $x^{-i}$, the player $i$'s problem in price-taking formulation has a structure of the problem $P^t_i(p,x^{-i})$ which is given as follows:
$$
\maxproblem{$P^t_i(p,x^{-i})$}
{x_i \in X_i}
{h(p,x)+g_i(x).}
{x_i \in {X}_i. }
$$
Similar to the case of price-anticipative formulation, here we use $t$ in $P^t_i$ to denote that the problem is player $i$'s problem in the price-\textit{taking} formulation. Even though the price $p$ is a given parameter for a player in the game, it is determined as a solution of the price-determining player problem $\mathcal{S}(x)$. We use $\mathcal{T}$ to denote this price-coupling game in price-taking formulation.
We define a set $\mathcal{A}$ as, $$\mathcal{A} \triangleq \{(x,p)|x \in X, p \in \mathcal{M}(x)\}.$$  Now we define an equilibrium for the game $\mathcal{T}$ as follows.

\begin{definition}\label{def:price_taking}
	A point $(x^*,p^*) \in \mathcal{A}$ is said to be an equilibrium of the game $\mathcal{T}$ if the following inequalities are satisfied $\forall i \in \mathcal{N}$,
	$$h(p^*,x^*)+g_i(x^*) \geq h(p^*,x_i,x^{-i*})+g_i(x_i,x^{-i*}), \forall x_i \in X_i,$$
	$$f(x^*,p^*) \leq f(x^*,p), \quad \forall p \in \mathcal{M}(x^*).$$
\end{definition}

As seen from the above definition, an equilibrium of the game $\mathcal{T}$ is defined an equilibrium of a game with $N+1$ players. Thus we call this formulation also as an $N+1$ player formulation. The additional player in this game is the price-determining player who decides the equilibrium price $p^*$. In the following sections, we study the existence of equilibria for the game $\mathcal{T}$ by categorizing into two classes depending on the presence or absence of the decision of players $x$ in the constraint set $\mathcal{M}$ of the price-determining player. In the first class of games, we consider the case where the constraint set $\mathcal{M}$ is independent of the decision of players. We denote this game as $\mathcal{T}_1$. In the second class of games, we consider the general case where the decision of players $x$ also constrains the feasible set $\mathcal{M}(x)$ of the price-determining player problem. 
\subsection{Game $\mathcal{T}_1$ : $\mathcal{M}$ is independent of $x$. }
Here we consider a price-coupling game in the price-taking formulation with the player $i$'s problem given by $P^t_i(p,x^{-i})$ and the price-determining player's problem is given by the following problem denoted by $\bar{\mathcal{S}}(x)$. 
$$
\problem{$\bar{\mathcal{S}}(x)$}
{p \in \mathcal{M}}
{f(x,p).}
{p \in \mathcal{M}}
$$
Note that the constraint of the price-determining player problem is independent of $x$. The Bertrand model of competition for suppliers in an electricity market is an example of such games. The suppliers' bids are the price of electricity that they are willing to supply which do not appear as a constraint in the ISO's problem. These games are easier to analyze because they do not have a coupled constraint structure as in the general case. The following lemma gives a result for the existence of an equilibrium in such class of games.

\begin{proposition} \label{lemma:T_1}
	Consider the game denoted by $\mathcal{T}_1$. Suppose the sets $X_i, i \in \mathcal{N}$ and the set valued function $\mathcal{M}$ are non-empty, convex and compact subset of some Euclidean space. Let the functions $f, h+g_i, i\in \mathcal{N}$ be continuous and $f(x,\cdot)$ be convex for all $x$ and $\forall i\in \mathcal{N}, h(p,\cdot,x^{-i})+g_i(\cdot,x^{-i})$ is concave for all $p,x^{-i}$, then there exists an equilibrium for the game $\mathcal{T}_1$. 
\end{proposition}
\begin{proof}
	This proposition can be verified using the results for the existence of an equilibrium in a classical $N$- player game \cite{basar99dynamic}. 
\end{proof}
Thus by checking the conditions in Proposition $\ref{lemma:T_1}$, we can guarantee the existence of an equilibrium for the game $\mathcal{T}_1$. But this result cannot be extended for a general class of games when the constraint is coupled which is described in the next section.

\subsection{Game $\mathcal{T}$: $\mathcal{M}$ dependent on $x$.}
Now we consider the general class of price-coupling games in the price-taking formulation where $\mathcal{M}$, the constraint set of the price-determining player is dependent on $x$, the decision of players. To provide a result for the existence of equilibria in such class of games, we consider a modification of the original game $\mathcal{T}$. The modification is that the constraint of the price-determining player problem is included as a constraint to the players' problem. We denote the game with this modification as ${\mathcal{T}_m}$ and the problem of player $i$ in $\mathcal{T}_m$ as ${\bar{P}}_i^t(p,x^{-i})$ which is given as follows: 
$$
\maxproblem{${\bar{P}}_i^t(p,x^{-i})$}
{x_i}
{h(p,x)+g_i(x)}
{x_i \in {X}_i, p \in \mathcal{M}(x).}
$$
We denote the feasible set of the player $i$'s problem ${\bar{P}}_i^t(p,x^{-i})$ as $\Gamma_i(p,x^{-i})$ which is given as follows:
\begin{equation}
\Gamma_i(p,x^{-i})= \{x_i | x_i \in X_i, p \in \mathcal{M}(x)\}.
\end{equation}
In this section, we show that an equilibrium of the original game $\mathcal{T}$ is also an equilibrium of the modified game $\mathcal{T}_m$. We also show that the reverse relation of equilibria, that is an equilibrium of game $\mathcal{T}_m$ is also an equilibrium of $\mathcal{T}$ under certain conditions.
The price-determining player problem $\mathcal{S}(x)$ is the same as the one defined in Section \ref{sec:introduction}. An equilibrium of the game $\mathcal{T}_m$ is defined as follows.
\begin{definition}
	Consider the game $\mathcal{T}_m$. A point $(\hat{x},\hat{p}) \in \mathcal{A}$ is said to be an equilibrium of $\mathcal{T}_m$ if $\forall i \in \mathcal{N}$,
	$$h(\hat{p},\hat{x})+g_i(\hat{x}) \geq h(\hat{p},x_i,\hat{x}^{-i})+g_i(x_i,\hat{x}^{-i}), \forall x_i \in \Gamma_i(\hat{p},\hat{x}^{-i})$$
	$$f(\hat{x},\hat{p}) \leq f(\hat{x},p), \quad \forall p \in \mathcal{M}(\hat{x}). $$
\end{definition}

A relation between an equilibrium of the original game $\mathcal{T}$ and an equilibrium of the modified game $\mathcal{T}_m$ is given in the following propositions.  
\begin{proposition}\label{lemma:positive}
	Consider the game $\mathcal{T}$. Suppose $(x^*,p^*) \in \mathcal{A}$ is an equilibrium of the game $\mathcal{T}$, then $(x^*,p^*) \in \mathcal{A}$ is also an equilibrium of the game $\mathcal{T}_m$. 
\end{proposition} 
{We omit the proof owing to space constraints.}
%
The following proposition  relates an equilibrium of the modified game $\mathcal{T}_m$ to that of the original game $\mathcal{T}$.
\begin{proposition} \label{lemma:reverse_relation}
	Consider the game denoted by $\mathcal{T}_m$. Suppose $(\hat{x},\hat{p}) \in \mathcal{A}$ is an equilibrium of the game $\mathcal{T}_m $. Suppose $\forall i \in \mathcal{N}, \Gamma_i(\hat{p},\hat{x}^{-i}) \equiv X_i$, then $(\hat{x},\hat{p}) \in \mathcal{A}$ is also an equilibrium of the game $\mathcal{T}$.  
\end{proposition} 
\begin{proof}
	Suppose $(\hat{x},\hat{p}) \in \mathcal{A}$ is an equilibrium of the game $\mathcal{T}_m$. Since $\forall i \in \mathcal{N}, \Gamma_i(\hat{p},\hat{x}^{-i}) \equiv X_i$, $(\hat{x},\hat{p}) \in \mathcal{A}$ is also feasible as an equilibrium of the game $\mathcal{T}$. Suppose $(\hat{x},\hat{p}) \in \mathcal{A}$ is \textit{not} an equilibrium of the  game $\mathcal{T}$. Then either,
	for some $i \in \mathcal{N}$, $\exists$ an $\bar{x}_i \in X_i$ such that,
	$${h(\hat{p},\bar{x}_i,\hat{x}^{-i})+g_i(\bar{x}_i,\hat{x}^{-i})} >  {h(\hat{p},\hat{x})+g_i(\hat{x})},$$ or
	$\exists$ a $\bar{p} \in \mathcal{M}(\hat{x})$ such that,
	$f(\hat{x},\bar{p}) < f(\hat{x},\hat{p}).$
	Since $\forall i \in \mathcal{N}, \Gamma_i(\hat{p},\hat{x}^{-i}) \equiv X_i$, $\bar{x}_i \in X_i \implies \bar{x}_i \in \Gamma_i (\hat{p},\hat{x}^{-i})$. In either of these cases, $(\hat{x},\hat{p}) \in \mathcal{A}$ cannot be an equilibrium of the  game $\mathcal{T}_m$ which is a contradiction to the assumption. Thus an equilibrium of the game $\mathcal{T}$ is also an equilibrium of the game $\mathcal{T}_m$ if $\Gamma_i(\hat{p}, \hat{x}^{-i})=X_i, \forall i \in \mathcal{N}$.
\end{proof}
Thus, an equilibrium of the game $\mathcal{T}_m$ is an equilibrium of the game $\mathcal{T}$ if $\forall i \in \mathcal{N}, \Gamma_i(\hat{p},\hat{x}^{-i}) \equiv X_i$, where $\hat{p}$ and $\hat{x}^{-i}$ are the equilibrium price and rivals action at the equilibrium of $\mathcal{T}_m$. 
In relation to an electricity market, the condition  $X_i = \Gamma_i({p}, {x}^{-i})$ says that for player $i \in \mathcal{N}$, given $x^{-i}$, the price $p$ that the ISO determines is valid for all the bids of player $i$.
Now we use the relations between the games $\mathcal{T}$ and $\mathcal{T}_m$ to provide a condition for the existence of equilibria for a game $\mathcal{T}$ having a potential function.  To make the notations simple, we denote the objective function of player $i$ as $J_i(x,p) \triangleq {h(p,x)+g_i(x)}$.  	
For a game $\mathcal{T}_m$ which has a potential function $\Pi$ \cite{monderer96potential}, consider the following optimization problem denoted by $\mathbb{P}^t$. 
$$
\maxproblem{$\mathbb{P}^t$}
{(x,p)\in \mathcal{A}}
{\Pi(x,p).}
{(x,p) \in \mathcal{A}}
$$
Lemma $\ref{lemma:price_taking_result}$ gives a relation of the maximizer of the problem $\mathbb{P}^t$ and an equilibrium of the  game $\mathcal{T}_m$.
\begin{lemma}{\label{lemma:price_taking_result}}
	Consider the game $\mathcal{T}_m$. Suppose $\Pi$ is a potential function for this game. If $(x^*,p^*) \in \mathcal{A}$ is a maximizer of problem $\mathbb{P}^t$, then $(x^*,p^*) \in \mathcal{A}$ is an equilibrium of this game.
\end{lemma}

\begin{proof}
	Suppose $(x^*,p^*) \in \mathcal{A}$ is a maximizer of the problem $\mathbb{P}^t$. Then,
	\begin{align}{\label{eq:inequality}}
	\Pi(x^*,p^*) \geq \Pi(x,p), \quad \forall (x,p) \in \mathcal{A}.
	\end{align}
	The inequality (\ref{eq:inequality}) still holds even if we replace $x$ by $x^*$ in the above condition. Hence,
	$\Pi(x^*,p^*) \geq \Pi(x^*,p), \forall (x^*,p) \in \mathcal{A}.$
	Since $(x^*,p) \in \mathcal{A} \implies p \in \mathcal{M}(x^*)$, the inequality can be rewritten as,
	$\Pi(x^*,p^*) \geq \Pi(x^*,p), \quad \forall p \in \mathcal{M}(x^*).$
	Since $\Pi$ is a potential function of the game $\mathcal{T}_m$,
	\begin{equation} \label{eq:equilibrium1}
	f(x^*,p^*) \leq f(x^*,p), \quad \forall p \in \mathcal{M}(x^*).
	\end{equation}
	Similarly, the inequality (\ref{eq:inequality}) holds even if we replace $p$ by $p^*$ and $x^{-i}$ by $x^{-i*}$. That is,
	$$\Pi(x^*,p^*) \geq \Pi(x_i,x^{-i*},p^*), \quad \forall (x_i,x^{-i*},p^*) \in \mathcal{A}.$$
	Since	$(x_i,x^{-i*},p^*) \in \mathcal{A} \implies x_i \in X_i, p^* \in \mathcal{M}(x_i,x^{-i*})$, the above condition can be rewritten as,
	$\Pi(x^*,p^*) \geq \Pi(x_i,x^{-i*},p^*), \quad \forall x_i \in \Gamma_i(p^*,x^{-i*}).$
	Since $\Pi$ is a potential function of the game $\mathcal{T}_m$, $\forall i \in \mathcal{N}$,
	\begin{align}\label{eq:equilibrium2}
	h(x^*,p^*)+g_i(x^*) \geq h(x_i,x^{-i*},p^*)+g_i(x_i,x^{-i*}), \quad \forall x_i \in \Gamma_i(p^*,x^{-i*}).
	\end{align}
	Thus, from the conditions $\eqref{eq:equilibrium1}$ and $\eqref{eq:equilibrium2}$ it can be seen that $(x^*,p^*) \in \mathcal{A}$ is an equilibrium of the game $\mathcal{T}_m$. 
\end{proof}
Thus a maximizer of the optimization problem $\mathbb{P}^t$ is an equilibrium for the game $\mathcal{T}_m$. Now we use this optimization problem to provide conditions for the existence of an equilibrium in such games.

\begin{corollary} 
	Consider the game $\mathcal{T}_m$. Suppose this game admits a potential function $\Pi$. Then there exists an equilibrium if the function ${\Pi}$ is continuous and the set $\mathcal{A}$ is non empty and compact. The set  $\mathcal{A}$ is compact if the {set $X$ is compact and $\mathcal{M}$ is closed and uniformly bounded.}
\end{corollary}
Now by using the relation of equilibria between the games $\mathcal{T}$ and $\mathcal{T}_m$, we give conditions for the existence of an equilibrium in the original game $\mathcal{T}$.
\begin{corollary}
	Consider the game $\mathcal{T}$ which has potential function $\Pi$. Consider the optimization problem $\mathbb{P}^t$ for such a game. Suppose $(x^*,p^*) \in \mathcal{A}$ is a maximizer of the problem $\mathbb{P}^t$ and suppose $ \forall i \in \mathcal{N}, \Gamma_i({p}^*,{x}^{-i*})= X_i.$ Then $(x^*,p^*) \in \mathcal{A}$ is an equilibrium of the game $\mathcal{T}$.
\end{corollary}

\begin{proof}
	The maximizer of the problem $\mathbb{P}^t$ is an equilibrium of the game $\mathcal{T}_m$ which is given by $(x^*,p^*) \in \mathcal{A}$. Since $\forall i \in \mathcal{N}, \Gamma_i({p}^*,{x}^{-i*}) = X_i$, by Proposition $\ref{lemma:reverse_relation}$, this equilibrium is also an equilibrium of the game $\mathcal{T}$.
\end{proof}
Thus with these conditions the existence of equilibria can be guaranteed for a price-coupling game in the price-taking formulation. In the next section, we use the results from price-coupling games to analyze the existence of equilibria in generation expansion planning problem.

\section{Generation Expansion Planning Problem in Electricity Markets}\label{sec:GEP}
We now apply the above results to the generation expansion planning problem in electricity markets. 
An efficiently functioning electricity market always makes sure that there is enough generation capacity available to meet the future projected demand. To meet the future demand, new generating units need to be added to the existing power system. Generation expansion planning (GEP) problem addresses the problem of finding the plant, the technology, the plant capacity, the location of the plant and  other decisions related to generation expansion that ensures that the capacity meets the demand in an economic fashion.

In the traditional centralized setting, the central power system planner performs a cost-minimization optimization to find these parameters related to generation expansion.  In a deregulated market, multiple generation companies make independent investment plans in order to to maximize their profit. This leads to strategic interaction and gaming among the companies involved in generation expansion. Hence the decision of a company affects other companies profits and decisions. In this section, we apply the results developed in the previous section to the GEP problem discussed in \cite{chuang2001game}. We consider a dynamic version of the problem over a discrete time for horizon $T$. In particular, we focus on open loop equilibria in such games (see e.g.,~\cite{abraham2017newresults}). In a GEP problem, the decision variables of generation company $i$ are the following:
\begin{itemize}
	\item $x_{ci}$, the plant capacity in MW, \item $x_{oti}$, the operating reserve at time instant $t$ in MW, 
	\item $e_{ti}$,  the quantity of power in MW to be used in the energy market at time instant $t$,
	\item $r_{ti}$, the quantity of power in MW to be used in the real-time market at time $t$.
\end{itemize} 
Using the decision variables submitted by the generation companies, the ISO declares the following prices: 

\begin{itemize}
	\item $P_t$, the price of $X_{ot}$ reserve, where $X_{ot} = \sum_{i=1}^{N} x_{oti}$, \item $MCP_t$, the fuel cost of marginal unit in energy market at time $t$, 
	\item $RTP_t$, the fuel cost of marginal unit in real time market at time $t$. 
\end{itemize}
The other notation used in this problem formulation is given as follows:
\begin{itemize}
	\item $T$, the time duration of planning period,
	\item $F_{i}$, the production cost of unit $i$,
	\item $C_i$, the annualized capital investment cost of unit $i$,
	\item $\rho_i$, the forced outage rate of unit $i$, 
	\item $\rho$, the average forced outage rate, 
	\item $L_t$, the load at time $t$, 
	\item $a$, the system wide average customer outage cost, 
	\item $ELDC$, the equivalent load duration curve of the system, 
	\item $C_0$, the total capacity of the existing system, 
	\item $R^u_i$ and $R^d_i$, the ramp-up and ramp-down limits of generating unit $i$.
\end{itemize}
Given other players decisions denoted by $x^{-i}_{c}, x^{-i}_{ot}, e^{-i}_{t}, r^{-i}_{t}, t=1:T$, player $i$'s problem is given as follows \cite{chuang2001game}:
$$
\maxproblem{}
{x_{ci}, x_{oti}, e_{ti}, r_{ti}, t=1:T}
{ \begin{aligned} \sum\limits_{t=1}^T [ MCP_t. e_{ti} +  RTP_t. r_{ti} + P_t. x_{oti} ]  - [ C_i(x_{ci}) + \sum\limits_{t=1}^T   F_{i}(e_{ti}+r_{ti} )]  \end{aligned} }
{{\begin{array}{l@{\ }c@{\ }l}
		0 \leq x_{oti} \leq x_{ci}, \forall t, \\
		e_{ti} \leq x_{ci} - x_{oti}, \forall t,  \\
		r_{ti} \leq x_{ci} - e_{ti}, \forall t, \\
		R^d_i \leq ( e_{(t+1)i} + r_{(t+1)i} ) - ( e_{ti} + r_{ti} ) \leq R^u_i, \forall t\\
	\begin{aligned}	(P_t, MCP_t,  RTP_t, t=1:T) \in  {\rm{ISO}}(x_{ci}, x_{oti}, e_{ti}, r_{ti}, t=1:T, \forall i) \end{aligned}
		\end{array}}}
$$
In the above problem, a generation company maximizes its profit, which is the difference of the revenue  $\sum\limits_{t=1}^T [ MCP_t. e_{ti} + RTP_t. r_{ti} + P_t. x_{oti} ]$, and the cost involved ($C_i(x_{ci}) + \sum\limits_{t=1}^T   F_{i}(e_{ti}+r_{ti})$). The constraints in the above optimization problem denote the operational limitations. The first constraint makes sure that the operating reserve at a time instant do not exceed the plant capacity. The second constraint guarantees that the quantity after committing for the operation reserve is available for the energy market. The third constraint gives the quantity available for usage in the real-time market. The fourth constraint  guarantees that the ramp-up and ramp-down constraints of a generating unit are not violated.

The ISO determines the operating reserve price $P_t$, the energy price $MCP_t$ and the real-time price $RTP_t$ at time $t$ in separate markets which are given as follows:
\begin{enumerate}
	\item Operating reserve price, $P_t$:
	The ISO sums up the operating reserve submitted by all the players at time $t$ given by $X_{ot} = \sum_{i=1}^{N} x_{oti}, \forall t$. The price for the reserve is calculated using the following equation, 
	$P_t = T (1- \rho ) ELDC(C_0+X_{ot}) a, \forall t$ (see \cite{chuang2001game} for more details).
	
	\item Energy price, $MCP_t$:
	The energy price is calculated in the energy market (e.g. Day-ahead Market).  Marginal unit is the unit at which the load demand is met, which is given by the following condition:
	$$\sum\limits_{i=1}^N e_{ti} = L_t, \forall t.$$
	Energy price is calculated as the fuel cost of the marginal unit in the energy market at time $t$.
	
	\item Real-time price, $RTP_t$:
	The real-time price is calculated in real-time market. This price is given by the fuel cost of the marginal unit in the real-time market at time $t$ and the marginal unit is that unit which satisfies the following condition in real-time market.
	$$\sum\limits_{i=1}^N r_{ti} 
	= \sum\limits_{i=1}^N \rho_i.e_{ti}, \forall t$$
\end{enumerate}

Two additions from the GEP model presented in \cite{chuang2001game} are the following. The ramp rate limit constraints are not considered in the model of GEP problem in \cite{chuang2001game}. This makes the problem dynamic in nature and hence more complicated to solve. Also, we consider the problem of ISO as a follower-level problem, while in \cite{chuang2001game}, it is considered in a generation company's problem itself. It can be observed that the GEP problem that we present fits into a form of the price-coupling game. Table \ref{Table1} gives a summary of the comparison of decision variables, the objective function and the feasible region of the price-coupling game and the GEP problem. We use the notation $MUP()$ to denote the marginal unit price at which the condition in the brackets is satisfied, that is, $MCP_t = MUP(\sum\limits_{i=1}^N e_{ti} = L_t),$ and $RTP_t = MUP(\sum\limits_{i=1}^N r_{ti} 
= \sum\limits_{i=1}^N \rho_i.e_{ti}).$ It can be noted that in the GEP problem, the price is a vector. In the next section, we analyze the existence of equilibrium in the price-taking formulation, and later, the price-anticipative formulation of the GEP problem.
\begin{table}
	\normalsize
	\setlength\tabcolsep{0.1cm}
	\def\arraystretch{2.0}
	\begin{center}
		\begin{tabular}{|p{1.5cm} ||p{3.2cm} | p{7cm} |} 
			\hline
			& Price-Coupling Game & Generation Expansion Planning Problem \\  
			\hline\hline
			Decision Variables &
			${{\begin{array}{l@{\ }c@{\ }l}
					1. \quad x_i \\
					2. \quad p_i	\end{array}}}$
			&
			${{\begin{array}{l@{\ }c@{\ }l}
					1. \quad (x_{ci}, x_{oti},e_{ti},r_{ti}, t=1:T)	\\
					2.\quad (P_{ti}, MCP_{ti}, RTP_{ti}, t=1:T)	\end{array}}}$				  \\ 
			\hline
			Objective Function Terms &
			${{\begin{array}{l@{\ }c@{\ }l}
					h(p_i,x_i)+g_i(x_i)\\
					1. \quad h(p_i,x_i) = p_i.x_i\\
					2. \quad g_i(x_i)
					\end{array}}}$
			&
			${{\begin{array}{l@{\ }c@{\ }l}
					1. \quad \sum\limits_{t=1}^T [ MCP_{ti}. e_{ti} + RTP_{ti}. r_{ti} +  P_{ti}. x_{oti}] \\
					2. \quad - [ C_i(x_{ci}) + \sum\limits_{t=1}^T   F_{i}(e_{ti}+r_{ti} )]  \end{array}}}$				  \\ 
			\hline			
			Feasible Region &
			${{\begin{array}{l@{\ }c@{\ }l}
					1. \quad x_i \in X_i \\
					2. \quad p_i \in SOL[\mathcal{S}(x)]
					\end{array}}}$
			&
			${{\begin{array}{l@{\ }c@{\ }l}
					1. \quad 0 \leq x_{oti} \leq x_{ci}, \forall t, \\
					e_{ti} \leq x_{ci} - x_{oti}, \forall t,  \\
					r_{ti} \leq x_{ci} - e_{ti}, \forall t, \\
					R^d_i \leq (e_{(t+1)i} + r_{(t+1)i} ) - ( e_{ti} + r_{ti} ) \leq R^u_i, \forall t \\
					2. \quad \text{(i) Operating Reserve Price} P_t\\
					
					X_{ot} = \sum_{i=1}^{N} x_{oti}, \forall t,\\
					P_t = T (1- \rho) ELDC(C_0+X_{ot}) a, \forall t. \\
					\text{(ii) Energy Market Price} MCP_t \\
					MCP_t = MUP(\sum\limits_{i=1}^N e_{ti} = L_t), \forall t \\\text{(The fuel cost of the marginal unit)} \\
					\text{(iii) Real-time Price} RTP_t \\
					RTP_t = MUP(\sum\limits_{i=1}^N r_{ti} 
					= \sum\limits_{i=1}^N \rho_i.e_{ti}), \forall t \\\text{(The fuel cost of the marginal unit)}
					\end{array}}}$				  \\ 
			\hline
		\end{tabular}
		\vspace{0.2cm}
		\caption{Comparison of the Price-Coupling Game and the GEP problem}
		\label{Table1}
	\end{center}
\end{table}

\subsection{Price-taking formulation of the GEP Problem}
The price-taking formulation is modelled as an $N+1$ player game, with ISO as the additional player in the game.  The generation companies in GEP problem take their decision considering prices as a given parameter. Hence, given $(P_t, MCP_t, RTP_t, t=1,\hdots,T) $ and the other players decisions $(x^{-i}_{c}, x^{-i}_{ot}, e^{-i}_{t}, r^{-i}_{t}, t=1,\hdots,T)$, a generation company $i$'s problem in the price-taking formulation is given as follows:
$$
\maxproblem{}
{x_{ci}, x_{oi}, e_{ti}, r_{ti}, t=1:T}
{
	\begin{aligned}	
	\sum\limits_{t=1}^T [ MCP_t. e_{ti} + RTP_t. r_{ti} + P_t. x_{oti} ]  -[ C_i(x_{ci}) + \sum\limits_{t=1}^T   F_{i}(e_{ti}+r_{ti} )]   
	\end{aligned}	
}
{{\begin{array}{l@{\ }c@{\ }l}
		0 \leq x_{oti} \leq x_{ci}, \forall t, \\
		e_{ti} \leq x_{ci} - x_{oti}, \forall t,  \\
		r_{ti} \leq x_{ci} - e_{ti}, \forall t, \\
		R^d_i \leq ( e_{(t+1)i} + r_{(t+1)i} ) - ( e_{ti} + r_{ti} ) \leq R^u_i, \forall t
		\end{array}}}
$$
The additional  player, ISO determines the different prices from the price curves using the decisions of the generation companies as follows:
(i) Operating reserve price $P_t = T (1- \rho) ELDC(C_0+X_{ot}) a, \forall t, X_{ot} = \sum_{i=1}^{N} x_{oti}, \forall t$, 
{(ii) Energy market price,}
$MCP_t = MUP(\sum\limits_{i=1}^N e_{ti} = L_t), \forall t$ and 
{(iii) Real time price,}
$RTP_t = MUP(\sum\limits_{i=1}^N r_{ti} 
= \sum\limits_{i=1}^N \rho_i.e_{ti}), \forall t.$ The following corollary provide conditions for the existence of an equilibrium for GEP problem in the price-taking formulation. We observe that this game is of type $\mathcal{T}_1$, using which we give our existence result.
\begin{corollary}
	Consider the generation expansion planning (GEP) game in the price-taking formulation with a generation company $i$'s problem given as above. If the cost functions $C_i, F_{i}, i=1:N$ are convex, then the game admits an equilibrium.
\end{corollary}

\begin{proof}
We observe that this game is of type $\mathcal{T}_1$. 
	If $C_i$ and $F_{i}$ are convex, then the objective functions of the generation companies are concave over a compact set. The ISO determines the prices from the price-curves, and not as a solution of an optimization problem. Hence by Proposition \ref{lemma:T_1}, the result follows.	
\end{proof}
In many cases, the cost functions take a convex quadratic form and hence the result can be applied to find an equilibrium in the  price-taking formulation. Next, we analyze the price-anticipative formulation of the GEP problem.

\subsection{Price-anticipative formulation of the GEP Problem}
In the GEP problem, the common term in the objective function of player $i$'s problem has the structure given by $h= p_i.x_i$, which is of the form of game $\mathcal{E}_2$. Recall from the price-coupling game section that player $i$'s problem in game $\mathcal{E}_2$ is given by the problem denoted by $P^{a\mathbf{2}}_i(x^{-i})$. For the price-coupling game $\mathcal{E}_2$, we define a related game called the price-coupling game with consistent price conjectures denoted by $\bar{\mathcal{E}}_2$ such that an equilibrium of $\mathcal{E}_2$ is also an equilibrium of $\bar{\mathcal{E}}_2$. Player $i$'s problem in the game $\bar{\mathcal{E}}_2$ is given by the problem denoted by $\bar{P}^{a\mathbf{2}}_i(x^{-i},p^{-i})$. For game $\bar{\mathcal{E}}_2$, we propose an optimization problem denoted by $\mathbb{P}^2$ such that a maximizer of this problem is an equilibrium of the game $\bar{\mathcal{E}}_2$. 
An equivalent of problem $\mathbb{P}^2$ in the GEP problem is given as follows: 
$$
\maxproblem{}
{}
{
	\begin{aligned}
	 \sum\limits_{i=1}^{N} \Bigl(   \sum\limits_{t=1}^T [ P_{ti}(X_{ot}) x_{oti} + & MCP_{ti}. e_{ti} +  RTP_{ti}. r_{ti} ] & - [ C_i(x_{ci}) + \sum\limits_{t=1}^T   F_{i}(e_{ti}+r_{ti} )] \Bigr) 
	 \end{aligned} }
{(x_{ci}, x_{oti},e_{ti},r_{ti},P_{ti}, MCP_{ti}, RTP_{ti}, t=1:T, i=1:N ) \in \mathcal{F}^2, }
$$
where $\mathcal{F}^2$ is given by the following. 
\begin{multline*}
\mathcal{F}^2 = \{(x_{ci}, x_{oti},e_{ti},r_{ti},P_{ti}, MCP_{ti}, RTP_{ti}, t=1:T, i=1:N )|  0 \leq x_{oti} \leq x_{ci}, \forall t, 
e_{ti} \leq x_{ci} - x_{oti}, \forall t, \\
r_{ti} \leq x_{ci} - e_{ti}, \forall t, 
R^d_i \leq ( e_{(t+1)i} + r_{(t+1)i} ) -  (e_{ti} + r_{ti} ) \leq R^u_i, \forall t, \\
X_{ot} = \sum_{i=1}^{N} x_{oti}, \forall t,
P_{ti} = T (1- \rho) ELDC(C_0+X_{ot}) a, \forall t, \\
MCP_{ti}= MUP(\sum\limits_{i=1}^N e_{ti} = L_t), \forall t,
RTP_{ti}= MUP(\sum\limits_{i=1}^N r_{ti} 
= \sum\limits_{i=1}^N \rho_i.e_{ti}), \forall t, \\
(P_{ti}, MCP_{ti},RTP_{ti})  = (P_{tj}, MCP_{tj},RTP_{tj}), \forall i,j \in \mathcal{N}, \forall t \}.
\end{multline*}
It can be observed that the term $\sum\limits_{i=1}^{N} - [ C_i(x_{ci}) + \sum\limits_{t=1}^T   F_{i}(e_{ti}+r_{ti} )]$ is a potential function of the player-dependent term in the objective function of player's problem in the GEP game. The following corollary shows that there exists an equilibrium for the  GEP game with consistent price conjectures.

\begin{corollary} (From Theorem \ref{theorem:potential})
	Consider the GEP game with consistent price conjectures. For such a game, consider the problem denoted by $\mathbb{P}^2$. A maximizer of the problem $\mathbb{P}^2$ is an equilibrium of the GEP game with consistent price conjectures.
\end{corollary}
It can be verified that the objective function of the problem $\mathbb{P}^{2}$ is continuous and the constraint set $\mathcal{F}$ is compact. Hence, there exists a maximizer for the above problem which is an equilibrium of the GEP game with consistent price conjectures. 

We considered the application of results of price-coupling games to generation expansion planning problem. Similar problem structures arise in various other market models and hence one can use the results developed in this paper to study the behaviour of the system and the participants in the market.

\section{Conclusion} \label{sec:conclusion}
In this paper, we study a class of games called the price-coupling games. The price-determining player is an important entity in such games which decide the price uniquely for all players as a function of the decision of the players. The existence of an equilibrium is not evident in such class of games. One reason is the coupling of the constraint of players through the price in the system. We formulate these games as price-anticipative formulation and price-taking formulation. In this paper, we provide conditions for the existence of equilibria for different categories of price-coupling games in these formulations. In particular, we analyze the generation expansion planning problem using the results developed in this paper.

%
\IEEEpeerreviewmaketitle

\ifCLASSOPTIONcaptionsoff
  \newpage
\fi



\bibliographystyle{IEEEtran}
\bibliography{IEEEabrv,../../../../Mylatexfiles/ref}

\end{document}